\documentclass[a4paper,10pt]{amsart}
\usepackage{amsfonts, amsmath, amsthm, amssymb, bbm}

\usepackage[cp1251]{inputenc}
\usepackage[russian, english]{babel}
\usepackage{xcolor,graphics}

\newcommand{\E}{\mathbb{E}}
\renewcommand{\P}{\mathbb{P}}
\newcommand{\R}{\mathbb{R}}
\newcommand{\bx}{\mathbf{x}}
\newcommand{\bt}{\mathbf{t}}
\newcommand{\by}{\mathbf{y}}
\newcommand{\be}{\mathbf{\eta}}

\newcommand{\diag}{\mathop{\mathrm{diag}}\nolimits}

\newtheorem{theorem}{Theorem}[section]

\newtheorem{lemma}[theorem]{Lemma}
\newtheorem{corollary}[theorem]{Corollary}
\newtheorem{proposition}[theorem]{Proposition}

\title{Correlation functions of real zeros of random polynomials}

\author[F.~G\"otze]{Friedrich G\"otze}
\address{Friedrich G\"otze, Faculty of Mathematics,
Bielefeld University,
P. O. Box 10 01 31,
33501 Bielefeld, Germany}
\email{goetze@math.uni-bielefeld.de}

\author[D.~Kaliada]{Dzianis Kaliada}
\address{Dzianis Kaliada, Institute of Mathematics, National Academy of Sciences of Belarus, 220072 Minsk, Belarus}
\email{koledad@rambler.ru}

\author[D.~Zaporozhets]{Dmitry Zaporozhets}
\address{Dmitry Zaporozhets\\
St.\ Petersburg Department of
Steklov Institute of Mathematics,
Fontanka~27,
 191011 St.\ Petersburg,
Russia}
\email{zap1979@gmail.com}

\thanks{Supported by CRC 701, Bielefeld University (Germany).}

\begin{document}

\begin{abstract}
We give an explicit formula for the correlation functions of real zeros of a random polynomial with arbitrary independent continuously distributed coefficients.
\end{abstract}

\maketitle

\section{Introduction}
Let $\xi_0,\xi_1,\dots,\xi_{n}$ be independent random variables with probability density functions $f_0,\dots,f_n$. Consider a random polynomial
\[
G(x)=\xi_nx^n+\xi_{n-1}x^{n-1}+\dots+\xi_1x+\xi_0,\quad x\in\R^1.
\]
With probability one, all zeros of $G$ are simple. Denote by $\mu$ the empirical measure counting the real zeros of $G$:
$$
\mu=\sum_{x:G(x)=0}\delta_x,
$$
where $\delta_x$ is the unit point mass at $x$. The distribution of $\mu$ can be described by its \emph{correlation functions} (also known as \emph{joint intensities}). Recall (see, e.g.,~\cite{HKPV09}) that the correlation functions of $\mu$ are functions (if well-defined) $\rho_k\,:\,\R^k\to\R^+$ for $k=1,\dots,n$, such that for any family of mutually disjoint Borel subsets $B_1,\dots,B_k\subset\R^1$,
$$
\E\,\left[\prod_{i=1}^{k}\mu(B_i)\right]=\int_{B_1}\dots\int_{B_k}\,\rho_k(x_1,\dots,x_k)\,dx_1\dots dx_k.
$$
A standard tool for evaluating $\rho_k$ is the following extension of the Kac-Rice formula (see~\cite{BD97},\cite{BD04}):
\begin{equation}\label{1632}
\rho_k(x_1,\dots,x_k)=\int_{\R^k}\,|t_1\dots t_k| D_k(\mathbf{0},\bt,x_1,\dots,x_k)\,dt_1\dots dt_k,
\end{equation}
where $\bt=(t_1,\dots, t_k)$ and $D_k(\cdot,\cdot,x_1,\dots,x_k)$ is the joint probability density function of the random vectors
$$
(G(x_1),\dots,G(x_k))\quad\mathrm{and}\quad (G'(x_1),\dots,G'(x_k)).
$$

The goal of this paper is to provide more explicit expressions for $\rho_k(\bx)$. The main tool that we will use is the \emph{Coarea formula} (see Lemma~\ref{1110}).

Our methods can be applied to the case of dependent coefficients having arbitrary joint probability density function. For simplicity, we consider only the case of independent coefficients.

\section{Main result}
Let us start with some notation. Denote
$$\bx=(x_1,\dots, x_k)\in\R^k.
$$

We use the following notation for the elementary symmetric polynomials:
$$
\sigma_i(\bx) :=
\left\{
  \begin{array}{ll}
    \sum_{1\le j_1 < \dots < j_i\le k} x_{j_1} x_{j_2} \dots x_{j_i}, & \hbox{if}\quad 0\leq i\le k, \\
    0, & \hbox{otherwise.}
  \end{array}
\right.
$$

Denote by $V(\bx)$ the Vandermonde matrix
\[
V(\bx)=
\begin{pmatrix}
1 & x_1 & \dots & x_1^{k-1} \\
\vdots & \vdots & \ddots & \vdots \\
1 & x_k & \dots & x_k^{k-1} \\
\end{pmatrix}.
\]

To formulate the first result, consider the random function $\be=(\eta_0,\dots,\eta_{k-1})^T:\R^k\to\R^k$ defined as
\begin{equation}\label{2327}
  \mathbf{\eta}(\bx)=
  -V^{-1}(\bx)
\begin{pmatrix}
\sum_{j=k}^n\xi_jx_1^j\\
\vdots\\
\sum_{j=k}^n\xi_jx_k^j
\end{pmatrix}.
\end{equation}

\begin{theorem}\label{1203}
We have
\begin{align}\label{eq-rho-def}
\rho_k(\bx) &= \prod_{1\leq i<j\leq k}|x_i-x_j|^{-1}\\
&\times\E\left[\prod_{i=1}^{k}\left|\sum_{j=0}^{k-1}j\eta_j(\bx)x_i^{j-1}+\sum_{j=k}^{n}j\xi_j x_i^{j-1}\right|\prod_{i=0}^{k-1}f_i(\eta_i(\bx))\right]. \nonumber
\end{align}
\end{theorem}
Theorem~\ref{1203} has been stated in~\cite{dZ05} without  detailed proof.

It is possible to obtain an explicit expression for $\be(\bx)$ in terms of the \emph{Schur functions}. Recall that   for a partition $\lambda=(\lambda_1,\dots,\lambda_k)$ of length  $\leq k$  the Schur function $S_\lambda(\bx)$  is given by
\[
S_\lambda(\bx)=\frac{\det(x_i^{\lambda_{k-j+1}+j-1})_{1\leq i,j\leq k}}{\prod_{1\leq i<j\leq k}(x_j-x_i)}.
\]

\begin{proposition}\label{808}
For $i=1,\dots,k$, we have
\[
\eta_i(\bx) = (-1)^{k-i} \sum_{j=k}^n \xi_j S_{\lambda_{ij}}(\bx),
\]
where the partition $\lambda_{ij}$ is defined as
$$
\lambda_{ij} = (j-k+1, \underbrace{1,\dots,1}_{k-i-1}, \underbrace{0,\dots,0}_i).
$$
\end{proposition}
For the basic properties of the Schur functions see, e.g.,~\cite[Section 1.3]{iM98-Hall}.

Now we are ready to state our main result.
\begin{theorem}\label{2021}
We have
\begin{align*}\label{eq-rho-b}
\rho_k(\mathbf{x}) &= \prod_{1\le i < j \le k} |x_i - x_j|\\
&\times\int\limits_{\mathbb{R}^{n-k+1}} \prod_{i=0}^{n}f_i\left(\sum_{j=0}^{n-k}(-1)^{k-i+j}\sigma_{k-i+j}(\bx)t_j\right) \prod_{i=1}^k \left|\sum_{j=0}^{n-k} t_j x_i^j\right| \, dt_0\dots dt_{n-k}.
\end{align*}
\end{theorem}

\begin{corollary}
For $k=n$ we have
\begin{equation}\label{eq-rho-nn}
\rho_n(\bx) =
\prod_{1\le i < j \le n} |x_i - x_j|
\int_{-\infty}^\infty |t|^n\, \prod_{i=0}^n f_i\Big((-1)^{n-i}\sigma_{n-i}(\bx)t\Big)\,dt.
\end{equation}
\end{corollary}

\section{Uniformly distributed coefficients}
In algebraic number theory, random polynomials with independent and uniformly distributed on $[-1,1]$ coefficients are of special interest (see~\cite{dK14},\cite{fGdZ14},\cite{GKZ15}). Let us apply Theorem~\ref{2021} to this case.

Suppose that
$$
f_i=\frac12\mathbbm{1}[-1,1],\quad i=0,\dots,n.
$$
Then it follows from Theorem~\ref{2021} that
$$
\rho_k(\mathbf{x}) = 2^{-n-1}\prod_{1\le i < j \le k} |x_i - x_j|\int\limits_{D_{\bx}} \prod_{i=1}^k \left|\sum_{j=0}^{n-k} t_j x_i^j\right| \, dt_0\dots dt_{n-k},
$$
where the domain of integration $D_{\bx}$ is defined as
$$
D_{\bx}=\left\{(t_0,\dots,t_{n-k})\in\R^{n-k+1}\,:\,\max_{0\leq i\leq n}\left|\sum_{j=0}^{n-k}(-1)^{i-j}\sigma_{i-j}(\bx)t_j\right|\leq1\right\}.
$$
In particular,
$$
\rho_n(\mathbf{x}) =\frac{2^{-n}}{n+1}\cdot\frac{\prod_{1\le i < j \le k} |x_i - x_j|}{(\max_{0\leq i\leq n}|\sigma_i(\bx)|)^{n+1}}.
$$

\section{The $n$-point correlation function}
It follows from the properties of the correlation functions (see, e.g.,~\cite{HKPV09}) that
$$
\E\,\left[\mu(\R^1)(\mu(\R^1)-1)\dots(\mu(\R^1)-n+1)\right]=\int_{\R^n}\,\rho_n(\bx)\,d\bx.
$$
Since $\mu(\R^1)\leq n$, we can calculate the probability that all zeros of $G$ are real:
\begin{align*}
  \P(&\mu(\R^1)= n) = \frac1{n!}\int_{\R^n}\,\rho_n(\bx)\,d\bx
  \\
   &=\frac1{n!}\int_{\R^n}\,\prod_{1\le i < j \le n} |x_i - x_j|
\int_{-\infty}^\infty |t|^n\, \prod_{i=0}^n f_i\Big((-1)^{n-i}\sigma_{n-i}(\bx)t\Big)\,dt\,d\bx.
\end{align*}
This formula has been obtained earlier in~\cite{dZ04}.

Let us calculate $\rho_n$ for some specific distributions.
\subsection{Gaussian distribution}
Suppose that
\[
f_i(t) = \frac{1}{\sqrt{2\pi}v_i } \exp\left(- \frac{t^2}{2v_i^2}\right),\quad i=0,\dots,n.
\]
Using the formula for the $n$-th absolute moment of the Gaussian distribution, it follows from~\eqref{eq-rho-nn} that
$$
\rho_n(\bx) =\frac{\sqrt{2} \Gamma\left(\frac{n+1}{2}\right)}{(2\pi)^{n/2} v_0\dots v_n}
 \left(\sum_{i=0}^n\frac{\sigma_{n-i}^2(\bx)}{v_i^2}\right)^{-(n+1)/2}\prod_{1\le i < j \le n} |x_i - x_j|.
$$
In particular, for $v_i=v_j$ we have
$$
\rho_n(\bx) =\frac{\sqrt{2} \Gamma\left(\frac{n+1}{2}\right)}{(2\pi)^{n/2} }
 \left(\sum_{i=0}^n\sigma_{n-i}^2(\bx)\right)^{-(n+1)/2}\prod_{1\le i < j \le n} |x_i - x_j|.
$$


\subsection{Exponential distribution}
Suppose that
\[
f_i(t) =\exp(-t )\mathbbm{1}\{t\geq0\},\quad i=0,\dots,n.
\]
Then with probability one $G$ does not have positive real zeros. Hence $\rho_n(\bx)>0$ only if $\bx$ lies in the negative orthant $\R^k_-$. In this case we have $(-1)^i\sigma_i(\bx)\geq0$ and by some elementary transformations,~\eqref{eq-rho-nn} implies
$$
\rho_n(\bx)=n! \left(\sum_{i=0}^n (-1)^i\sigma_i(\bx)\right)^{-(n+1)}\prod_{1\le i < j \le n} |x_i - x_j|\mathbbm{1}\{\bx\in\R^k_-\}.
$$
Using the well-known identity
$$
\sum_{i=0}^n (-1)^i\sigma_i(\bx)=\prod_{i=1}^n(1-x_i),
$$
we obtain
$$
\rho_n(\bx)=n! \frac{\prod_{1\le i < j \le n} |x_i - x_j|}{\prod_{i=1}^n(1-x_i)^{n+1}}\mathbbm{1}\{\bx\in\R^k_-\}.
$$


\section{Proof of Theorem~\ref{1203}}
Obviously, $G(x_1)=\dots=G(x_k)=0$ if and only if
\begin{equation}\label{2317}
\begin{pmatrix}
1 & x_1 & \dots & x_1^{n} \\
\vdots & \vdots & \ddots & \vdots \\
1 & x_k & \dots & x_k^{n} \\
\end{pmatrix}
\begin{pmatrix}
\xi_0\\
\vdots\\
\xi_n
\end{pmatrix}
=\mathbf{0},
\end{equation}
which is equivalent to
$$
\begin{pmatrix}
x_1^k & x_1^{k+1} & \dots & x_1^{n} \\
\vdots & \vdots & \ddots & \vdots \\
x_k^k & x_k^{k+1} & \dots & x_k^{n} \\
\end{pmatrix}
\begin{pmatrix}
\xi_k\\
\vdots\\
\xi_n
\end{pmatrix}
=-V(\bx)
\begin{pmatrix}
\xi_0\\
\vdots\\
\xi_{k-1}
\end{pmatrix}.
$$
Recalling~\eqref{2327}, we obtain that~\eqref{2317} is equivalent to
$$
\be(\bx)=
 \begin{pmatrix}
\xi_0\\
\vdots\\
\xi_{k-1}
\end{pmatrix}.
$$
Denote by $J_{\be}(\bx)$ the Jacobian matrix of $\be$ at the point $\bx$.
\begin{lemma}\label{1200}
$$
\det J_{\be}(\bx) =
 -
\frac{\prod_{i=1}^{k}\left(\sum_{j=0}^{k-1}j\eta_j(\bx)x_i^{j-1} + \sum_{j=k}^{n}j\xi_j x_i^{j-1}\right)}{\prod_{1\leq i<j\leq k}(x_j-x_i)}.
$$
\end{lemma}
\begin{proof}
Differentiating
\[
V(\bx) \eta(\bx)=
  -
\begin{pmatrix}
x_1^k & x_1^{k+1} & \dots & x_1^{n} \\
\vdots & \vdots & \ddots & \vdots \\
x_k^k & x_k^{k+1} & \dots & x_k^{n} \\
\end{pmatrix}
\begin{pmatrix}
\xi_k\\
\vdots\\
\xi_n
\end{pmatrix},
\]
we obtain
\begin{align*}
  V(\bx)J_{\be}(\bx)&+\diag\left(\sum_{j=0}^{k-1}j\eta_j(\bx)x_1^{j-1},\dots,\sum_{j=0}^{k-1}j\eta_j(\bx)x_k^{j-1}\right) \\
  &=-\diag\left(\sum_{j=k}^{n}j\xi_j x_1^{j-1},\dots,\sum_{j=k}^{n}j\xi_j x_k^{j-1}\right).
\end{align*}
We finish the proof by taking the second term from the left hand side to the right hand side and using
$$
\det V(\bx)=\prod_{1\leq i<j\leq k}(x_j-x_i).
$$
\end{proof}

\begin{lemma}[Coarea formula]\label{1110}
Let $B\subset \R^k$ be a region. Let $u:B\to\R^k$ be a Lipschitz function and $h:\R^k\to\R^1$ be an $L^1$-function. Then
\[
\int\limits_{\R^k}\#\{\bx\in B\,:\, u(\bx)=\by\}\,h(\by)\,d\by=\int\limits_{B}|\det J_u(\bx)|\, h(u(\bx))\,d\bx,
\]
where $J_u(\bx)$ is the Jacobian matrix of $u(\bx)$.
\end{lemma}
\begin{proof}
See~\cite[pp. 243--244]{hF1969}.
\end{proof}

Let $B_1,\dots,B_k$ be a family of mutually disjoint Borel subsets  in $\R^1$. Denote $B=B_1\times\dots\times B_k\subset\R^k$. We have
\begin{align*}
\E\,\left[\prod_{i=1}^{k}\mu(B_i)\right]&=\E\,\#\{\bx\in B\,:\, \eta(\bx)=(\xi_0,\dots,\xi_{k-1})\}\\
&=\E\,\int\limits_{\R^k}\,\#\{\bx\in B\,:\, \eta(\bx)=\by\}f_0(y_0)\dots f_{k-1}(y_{k-1})\,d\by.
\end{align*}
Applying Lemma~\ref{1110} and Fubini's theorem to the right hand side, we obtain
\begin{align*}
\E\,\left[\prod_{i=1}^{k}\mu(B_i)\right]&=\int\limits_{B}\E\,|\det J_{\be}(\bx)|\, f_0(\eta_0(\bx))\dots f_{k-1}(\eta_{k-1}(\bx))\,d\bx.
\end{align*}
Now the proof of Theorem~\ref{1203} follows from Lemma~\ref{1200}.

\section{Proof of Theorem~\ref{2021}}\label{2207}
Theorem~\ref{1203} states that
\begin{align}\label{eq-rho-int}
\rho_k(\bx) &= \prod_{1\leq i<j\leq k}|x_j-x_i|^{-1}\\
&\times\int\limits_{\R^{n-k+1}}\prod_{i=1}^{k}\left|\sum_{j=0}^{n}ja_j x_i^{j-1}\right|\prod_{i=0}^{n}f_i(a_i)\, da_k da_{k+1}\dots da_n, \nonumber
\end{align}
where $a_0, \dots, a_{k-1}$ are functions of $a_k, \dots, a_n$ and $x_1, \dots, x_k$:
\begin{equation}\label{1638}
\begin{pmatrix}
a_0\\
\vdots\\
a_{k-1}
\end{pmatrix}=
  -V^{-1}(\bx)
\begin{pmatrix}
\sum_{j=k}^n a_j x_1^j\\
\vdots\\
\sum_{j=k}^n a_j x_k^j
\end{pmatrix}.
\end{equation}

To prove the theorem,  we will use the ideas from \cite[pp. 58--59]{dK12-Trudy} and \cite[Lemmas 5 and 6]{dK14}. Equation~\eqref{1638} means that $x_1,x_2,\dots,x_k$ are zeros of the polynomial $\sum_{j=0}^n a_j x^j$. Hence there exists a unique polynomial $\sum_{j=0}^{n-k} b_j x^j$ such that
\begin{equation}\label{1658}
\sum_{j=0}^n a_j x^j =\prod_{i=1}^k(x-x_i)\left(\sum_{j=0}^{n-k} b_j x^j \right)= \left(\sum_{j=0}^k (-1)^{k-j} \sigma_{k-j}(\bx) x^j\right) \left(\sum_{j=0}^{n-k} b_j x^j \right).
\end{equation}
The variables $a_0,\dots,a_n$ are uniquely defined by $\bx$ and $b_0,\dots,b_{n-k}$ from~\eqref{1658}:
\begin{equation}\label{1840}
a_i=\sum_{j=0}^{n-k}(-1)^{k-i+j}\sigma_{k-i+j}(\bx)b_j.
\end{equation}
Thus we can change variables in~\eqref{eq-rho-int} substituting $a_k,\dots,a_n$ by their expressions in terms of $\bx$ and $b_0,\dots,b_{n-k}$ from~\eqref{1840}. The Jacobian of this substitution is a lower triangle matrix with unities in the diagonal. Hence its determinant is equal to one.

Differentiating~\eqref{1658} at the point $x_i$ we get
\begin{equation}\label{1901}
\sum_{j=0}^{n}ja_j x_i^{j-1}=\prod_{j\ne i} (x_i - x_j) \left(b_{n-k} x_i^{n-k} + \dots + b_1 x_i + b_0\right),\quad i=1,\dots k.
\end{equation}
Substituting~\eqref{1840} and~\eqref{1901} in~\eqref{eq-rho-int} finishes the proof.

\section{Proof of Proposition~\ref{808}}
For $0\le i\le k-1$ and $j\ge k$, denote by $V^*_{ij}(\bx)$ the matrix obtained from $V(\bx)$ by substitution of the $i$-th column by $(x_1^j, \dots, x_k^j)^T$:
\[
V^*_{ij}(\bx)=
\begin{pmatrix}
1 & x_1 & \dots & x_1^{r-1} & x_1^j & x_1^{r+1} & \dots & x_1^{k-1} \\
\vdots & \vdots & & \vdots & \vdots & \vdots & & \vdots \\
1 & x_k & \dots & x_k^{r-1} & x_k^j & x_k^{r+1} & \dots & x_k^{k-1} \\
\end{pmatrix}.
\]
Then by Cramer's rule, we have
\[
\eta_r(\bx) = - \frac{1}{\det V(\bx)} \sum_{j=k}^n \xi_j \det V^*_{ij}(\bx).
\]
It is easily seen that
\[
\frac{V^*_{ij}(\bx)}{\det V(\bx)} = (-1)^{k-i-1} S_{\lambda_{ij}}(\bx),
\]
and the proof follows.

\bigskip

{\bf Acknowledgments.} The third named author is grateful to Ildar Ibragimov and Vlad Vysotsky for many useful discussions.

\bibliographystyle{plain}
\bibliography{corrf}

\end{document}